\documentclass[reqno, 10pt]{amsart}
\usepackage{amssymb,mathrsfs}
\usepackage[usenames, dvipsnames]{color}
\usepackage{esint}
\usepackage{hyperref}
\usepackage{todonotes}
\usepackage{verbatim}
\usepackage{enumerate}

\usepackage{cite}
\usepackage[nobysame,non-sorted-cites]{amsrefs}
\def\MR#1{}

\theoremstyle{plain}
\newtheorem{theorem}{Theorem}[section]
\newtheorem{lemma}[theorem]{Lemma}

\theoremstyle{definition}

\theoremstyle{remark}
\newtheorem{remark}[theorem]{Remark}

\newcommand{\dv}{\operatorname{div}}
\newcommand{\osc}{\operatorname{osc}}

\numberwithin{equation}{section}

\newcommand{\bR}{\mathbb{R}}

\newcommand{\bS}{\mathbb{S}}

\newcommand\cD{\mathcal{D}}

\makeatletter
\def\dashint{\operatorname%
{\,\,\text{\bf--}\kern-.98em\DOTSI\intop\ilimits@\!\!}}
\makeatother

\newcommand\dashnorm[2]{\nparallel\kern-.2em #1 \Vert_{#2}}

\begin{document}
\title[Insulated conductivity problem with partially flat inclusions]{Gradient estimates for the insulated conductivity problem with partially flat inclusions}

\author[H. Dong]{Hongjie Dong}

\author[Z. Yang]{Zhuolun Yang}

\author[H. Zhu]{Hanye Zhu}

\address[H. Dong]{Division of Applied Mathematics, Brown University, 182 George Street, Providence, RI 02912, USA}
\email{Hongjie\_Dong@brown.edu}

\address[Z. Yang]{Division of Applied Mathematics, Brown University, 182 George Street, Providence, RI 02912, USA; Department of Mathematics, The Ohio State University, 231 West 18th Avenue, Columbus, OH 43210, USA }
\email{yang.8242@osu.edu}

\address[H. Zhu]{Department of Mathematics, Duke University, 120 Science Drive, Durham, NC 27708, USA}
\email{hanye.zhu@duke.edu}

\thanks{H. Dong was partially supported by the NSF under agreement DMS-2350129.}
\thanks{Z. Yang was partially supported by the AMS-Simons Travel Grant and the NSF under agreements DMS-2510251.}
\thanks{H. Zhu was partially supported by the NSF under agreements DMS-2006372 and DMS-2306726.}



\begin{abstract}
We study the insulated conductivity problem with inclusions embedded in a bounded domain in $\bR^n$. It was known that in the setting of strictly convex inclusions, the gradient of solutions may blow up as the distance between inclusions approaches 0. The optimal blow-up rate was proved in \cites{DLY} and was achieved in the presence of a uniform background gradient field. In this paper, we demonstrate that when the inclusions are partially flat, the gradient of solutions does not blow up under any uniform background fields.
\end{abstract}

\maketitle
\section{Introduction and Main Results}
When two conducting or insulating inclusions are closely located in a composite, the gradient of the solution to the conductivity problem may become arbitrarily large in the narrow region in between them as the distance between the inclusions tends to zero. This phenomenon is known as field concentration, a central topic in the theory of composite materials. The study of field concentration phenomenon has attracted a lot of attention from various authors in recent decades (see, e.g. \cites{ACKLY,BASL,BLY2,BV,DLY,DYZ24,fukushima2024finiteness,li2024optimalgradientestimatesinsulated,LN,We,Y1} and the survey paper \cite{Kang}).

Let us first describe the mathematical setup of the conductivity problem. Let $n\ge 2$, $\Omega\subset \bR^n$ contain two disjoint open sets $\cD_1$ and $\cD_2$ with $\text{dist}(\cD_1\cup \cD_2, \partial\Omega)>1$. Let $\widetilde{\Omega} := \Omega \setminus \overline{(\cD_1 \cup \cD_2)}$ and $\varepsilon=\text{dist}(\cD_1, \cD_2)$. The conductivity problem can be modeled by the following elliptic equation:
\begin{equation}\label{general_problem}
\begin{cases}
-\mathrm{div}\Big((k \,\chi_{\cD_1\cup \cD_2}+ \chi_{\widetilde{\Omega}})(x)\,D{u}\Big)=0&\mbox{in}~\Omega,\\
u=\varphi(x)&\mbox{on}~\partial\Omega,
\end{cases}
\end{equation}
where $k$ represents the conductivity of the inclusions $\cD_1$ and $\cD_2$, $u$ represents the voltage potential, and $-D u$ stands for the electric field. For the case when $k$ is bounded away from $0$ and $\infty$, it was proved in \cites{BV,LV,LN} that $Du$ is always uniformly bounded independent of $\varepsilon$. However, if $k$ degenerates to either $\infty$ or $0$, the gradient of solutions may blow up when $\varepsilon\to 0$.

In the case when $k=\infty$, Eq.\eqref{general_problem} reduces to the following perfect conductivity problem 
\begin{equation*}
\begin{cases}
-\Delta u =0&\mbox{in}~\widetilde\Omega,\\
u  = U_j &\mbox{on}~\partial \cD_i,\, i = 1,2,\\
\int_{\partial \cD_i} \partial_\nu u \, d\sigma = 0& i=1,2,\\
u=\varphi&\mbox{on}~\partial\Omega,
\end{cases}
\end{equation*}
where $U_j$ is some constant determined by the third line.
When $\cD_1$ and $\cD_2$ are strictly convex sets, the optimal $L^\infty$ bound for $Du$ was shown in \cites{AKLLL,AKL,Y1,Y2,BLY1,BLY2} to be
\begin{equation*}
\begin{cases}
\| D u \|_{L^\infty(\widetilde{\Omega})} \le C\varepsilon^{-1/2} \|\varphi\|_{C^2(\partial \Omega)} &\mbox{when}~n=2,\\
\| D u \|_{L^\infty(\widetilde{\Omega})} \le C|\varepsilon \ln \varepsilon|^{-1} \|\varphi\|_{C^2(\partial \Omega)} &\mbox{when}~n=3,\\
\| D u \|_{L^\infty(\widetilde{\Omega})} \le C\varepsilon^{-1} \|\varphi\|_{C^2(\partial \Omega)} &\mbox{when}~n\ge 4.
\end{cases}
\end{equation*}
However, when
$\cD_1$ and $\cD_2$ have partially flat boundary near their closest points, it was proved in \cite{CLX} that $Du$ remains uniformly bounded independent of $\varepsilon$.  

In the case when $k=0$, Eq.\eqref{general_problem} becomes the following insulated conductivity problem
\begin{equation*}
\begin{cases}
-\Delta u =0&\mbox{in}~\widetilde\Omega,\\
\partial_\nu u  = 0 &\mbox{on}~\partial \cD_i, \,i = 1,2,\\
u=\varphi&\mbox{on}~\partial\Omega.
\end{cases}
\end{equation*}
When $\cD_1$ and $\cD_2$ are strictly convex sets in $\bR^2$, the optimal blow-up rate of $Du$ was shown to be of the order $\varepsilon^{-1/2}$ about two decades ago in \cites{AKL,AKLLL}. However, in higher dimensions $n\ge 3$, the problem becomes much more difficult. The optimal blow-up rate when the inclusions are strictly convex in dimensions $n\ge 3$ was only recently obtained in \cites{DLY,DLY2,li2024optimalgradientestimatesinsulated}, and is related to the principal curvatures of the inclusions. 

In this paper, we study the field concentration between two partially flat insulators. We assume that $\mathcal{D}_{1},\mathcal{D}_{2}\subset\Omega \subset \bR^n$ are axisymmetric with respect to the same axis, and $\cD_1$ and $\cD_2$ have partially flat boundaries near their closest points. We study the equation
\begin{equation}\label{equzero}
\left\{
\begin{aligned}
-\Delta u &=0 \quad \mbox{in }\widetilde{\Omega},\\
\partial_\nu u &= 0 \quad \mbox{on}~\partial\mathcal{D}_{i},~i=1,2,\\
 u &= \varphi \quad \mbox{on } \partial \Omega,
\end{aligned}
\right.
\end{equation}
where $\varphi\in{C}^{1,1}(\partial\Omega)$ is given and $\nu = (\nu_1, \ldots, \nu_n)$ denotes the inward normal vector on $\partial\mathcal{D}_{1} \cup \partial\mathcal{D}_{2}$.

We denote $x = (x', x_n)$, where $x' \in \mathbb{R}^{n-1}$, $B'_r:= \{x' \in \bR^{n-1} : |x'|< r\}$. By choosing a coordinate system properly, we can assume that near the origin, the part of $\partial \mathcal{D}_1$ and $\partial \mathcal{D}_2$, denoted by $\Gamma_+$ and $\Gamma_-$, are respectively the graphs of two $C^{2,\gamma}$ ($\gamma\in (0,1)$) functions in terms of $x'$. That is,
\begin{align*}
\Gamma_+ = \left\{ x_n = \frac{\varepsilon}{2}+h_1(x'),~|x'|<1\right\},\quad \Gamma_- = \left\{ x_n = -\frac{\varepsilon}{2}+h_2(x'),~|x'|<1\right\},
\end{align*}
where $h_1$ and $h_2$ are $C^{2,\gamma}$ radial functions satisfying 
\begin{equation}\label{fg_0}
h_1(x')=h_2(x')=0 \quad \mbox{in}~~B'_{r_0},
\end{equation}
\begin{equation}\label{fg_1}
h_1(x') - h_2(x') = a\,(|x'|-r_0)^2_+ + O \Big( (|x'|-r_0)^{2+\gamma}_+ \Big) \quad \mbox{in}~~B'_{1}
\end{equation}
for some positive constants $a$ and $r_0 < 1/2$. For $x_0 \in \widetilde\Omega$, $0 < r\leq 1$, we denote
\begin{align*}
\Omega_{x_0,r}:=\left\{(x',x_{n})\in \widetilde{\Omega}:~-\frac{\varepsilon}{2}+h_2(x')<x_{n}<\frac{\varepsilon}{2}+h_1(x'),~|x'|<1,~|x'-x_0' |<r\right\},
\end{align*}
and $\Omega_{r}:=\Omega_{0,r}$.
By the maximum principle and classical gradient estimates, the solution $u \in H^1(\widetilde\Omega)$ of \eqref{equzero} satisfies
\begin{equation}
\label{u_C1_outside}
\|u\|_{L^\infty(\widetilde\Omega)} + \| D u\|_{{L^\infty}(\widetilde\Omega \setminus \Omega_{1/2})} \le C\|\varphi\|_{C^{1,1}(\partial \Omega)}.
\end{equation}
Our main result is the uniform boundedness of the gradient in the presence of partially flat insulators under certain conditions on $\varphi$, including the cases of any uniform background fields.

\begin{theorem}\label{main-thm}
Let $h_1$, $h_2$ be $C^{2,\gamma}$ functions satisfying \eqref{fg_0} and \eqref{fg_1} with some $a > 0$, $r_0 \in (0, 1/2)$, and $\gamma\in(0,1)$. Let $n \ge 2$, $\varepsilon \in(0,1/4)$, and $u \in H^1(\widetilde\Omega)$ be the solution of \eqref{equzero} with 
\begin{equation}
    \label{eq17.07}
\varphi= \sum_{k=0}^M\sum_{i=1}^{N(k)} {f}_{k,i}(r,x_n)Y_{k,i}(\xi),
\end{equation}
where $M$ is a finite positive integer, $r = |x'|$, $Y_{k,i}(\xi)$ are the $L^2$-normalized spherical harmonics of degree $k$ on $\mathbb{S}^{n-2}$, $N(k)$ is the number of linearly independent spherical harmonics of degree $k$, and ${f}_{k,i}$'s are $C^{1,1}$ functions. Then there exists a constant $C>0$ depending only on $n$, $M$ $a$, $r_0$, $\gamma$, $\|h_1\|_{C^{2,\gamma}}$, $\|h_2\|_{C^{2,\gamma}}$, and $\text{diam}(\Omega)$, such that
\begin{equation}\label{main-eq}
\|D u\|_{L^\infty(\widetilde\Omega)} \le  C \sum_{k=0}^M\sum_{i=1}^{N(k)} \|{f}_{k,i}\|_{C^{1,1}}.
\end{equation}
\end{theorem}

\begin{remark}
Any $C^{1,1}$ function $\varphi$ can be expressed as an infinite sum of spherical harmonics:
$$
\varphi= \sum_{k=0}^\infty \sum_{i=1}^{N(k)} {f}_{k,i}(r,x_n)Y_{k,i}(\xi).
$$
In \eqref{eq17.07}, it is assumed that the boundary data $\varphi$ admits a spherical harmonics decomposition with only finitely many terms.
A physically significant example of boundary data is the uniform background field $\varphi = x_j$. Note that in the insulated conductivity problem with strictly convex inclusions, the optimal blow-up rate was achieved when $\varphi = x_j$ for $j = 1, 2, \ldots, n-1$ in \cites{DLY}, while for the perfect conductivity problem with strictly convex inclusions, the optimal blow-up rate was attained when $\varphi = x_n$ in \cites{BLY1}.
\end{remark}

\subsection{Difficulty and our strategy}\label{sec1.1}

We briefly explain the main difficulty of the problem in comparison to the setting of strictly convex insulators. An important ingredient in \cite{DLY} is the estimate \eqref{Du-bound} with $r_0=0$. The authors used the orthogonality of the spherical harmonics to obtain an optimal control of $\osc_{\Omega_{x,\eta}}u$ that appeared on the right-hand side of \eqref{Du-bound}. That argument is not applicable to the case when $r_0>0$, due to the fact that $|x'|$ and $\eta:= \frac{1}{4} (\varepsilon + (|x'|-r_0)_+^2)^{1/2}$ are not always comparable when $r_0 > 0$. Therefore, we rely on more refined and delicate estimates for each mode in the spherical harmonics decomposition \eqref{sph-decom}.

By the linearity of equation \eqref{equzero}, it suffices to prove the uniform boundedness of $Du$ in Theorem \ref{main-thm} for $\varphi$ consisting of a single mode. A key ingredient is Lemma \ref{lem:single_decom}, which shows that if the boundary data $\varphi$ contains only a single mode, then so does the solution $u$. We treat separately the cases when $u$ contains a single nonzero mode and when $u$ contains a single zero mode, and prove Theorem \ref{main-thm} in Sections \ref{sec3} and \ref{sec4}, respectively.

\section{Preliminary}
First, by following the proof of \cite{LY2}*{Theorem 1.1} (see also the proof of \cite{BLY2}*{Theorem 1.2}), one can use a similar flattening and extension argument 
to obtain 
\begin{lemma}
    Let $u \in H^1({\widetilde\Omega})$ be a solution of \eqref{equzero}. Then for $x\in \Omega_1$, and $\eta=(\varepsilon + (|x'|-r_0)_+^2)^{1/2}/4$, we have
\begin{equation}\label{Du-bound}
    |D u(x)| \le C \big(\varepsilon + (|x'|-r_0)_+^2 \big)^{-1/2}\underset{\Omega_{x,\eta}}{\osc}~u,
\end{equation}
where $C > 0$ depends only on $n$, $a$, $\gamma$, $\|h_1\|_{C^{2,\gamma}}$, and $\|h_2\|_{C^{2,\gamma}}$.
\end{lemma}
Without loss of generality, we assume that $a=1$. We perform a change of variables by setting
\begin{equation}\label{x_to_y}
\left\{
\begin{aligned}
y' &= x' ,\\
y_n &=   2 \varepsilon \left( \frac{x_n - h_2(x') + \varepsilon/2}{\varepsilon + h_1(x') - h_2(x')} - \frac{1}{2} \right) ,
\end{aligned}\right.
\quad \forall (x',x_n) \in \Omega_{1}.
\end{equation}
This change of variables maps the domain $\Omega_{1}$ to a cylinder of height $2\varepsilon$, denoted by $C_{0,1}$, where
\begin{equation}\label{C_x_s}
C_{x,s}:=\{ y = (y',y_n) \in \bR^n:~  |y'-x'| < s,\,  |y_n| < \varepsilon\}.
\end{equation}
for $s,t > 0$. Moreover, $\det (D_x y)=2\varepsilon(\varepsilon + h_1(x') - h_2(x'))^{-1}$.
Let $u$ be the solution and let $v(y) = u(x)$. Then $v$ satisfies
\begin{equation*}
\left\{
\begin{aligned}
-D_i(a^{ij}(y) D_j v(y)) &=0 \quad \mbox{in } C_{0,1},\\
a^{nj}(y) D_j v(y) &= 0 \quad \mbox{on } \{(y',y_n):~|y'|<1,\; y_n=\pm \varepsilon\},
\end{aligned}
\right.
\end{equation*}
with $\|v\|_{L^\infty(C_{0,1})} \le C$,
where the coefficient matrix $(a^{ij}(y))$ is given by
\begin{align}\label{a_ij_formula}
&(a^{ij}(y)) = \frac{ 2 \varepsilon(\partial_x y)(\partial_x y)^t}{\det (\partial_x y)} =\nonumber\\
& \begin{pmatrix}
\varepsilon + (|y'|-r_0)_+^2 &0 &\cdots &0 &a^{1n}\\
0 &\varepsilon + (|y'|-r_0)_+^2 &\cdots &0 &a^{2n}\\
\vdots &\vdots &\ddots &\vdots &\vdots\\
0 &0 &\cdots &\varepsilon + (|y'|-r_0)_+^2 &a^{n-1,n}\\
a^{n1} &a^{n2} &\cdots &a^{n,n-1} &  \frac{4\varepsilon^2 + \sum_{i=1}^{n-1}|a^{in}|^2}{\varepsilon + h_1(y') - h_2(y')}
\end{pmatrix}\nonumber\\
&+
\begin{pmatrix}
e^1 &0 &\cdots &0 &0\\
0 &e^2 &\cdots &0 &0\\
\vdots &\vdots &\ddots &\vdots &\vdots\\
0 &0 &\cdots & e^{n-1} &0\\
0 &0 &\cdots & 0 &0
\end{pmatrix},
\end{align}
and
\begin{equation}\label{a_ij_estimate}
\begin{aligned}
a^{ni} = a^{in} &= -2\varepsilon\, D_i h_2(y') - (y_n + \varepsilon)\,D_i(h_1(y') - h_2(y')) = O(1)\, \varepsilon \,(|y'|-r_0)_+;\\
e^{i} &= h_1(y')-h_2(y')-(|y'|-r_0)_+^2 = O\Big( (|y'|-r_0)_+^{2+\gamma} \Big)
\end{aligned}
\end{equation}
for $i = 1, \ldots , n-1$.

\begin{lemma}
Let $v$ be defined as above. Then we have
\begin{equation}\label{Dv_uniform_bound}
|Dv(y)| \le \frac{C}{r_0} , \quad y \in C_{0,3r_0/4},
\end{equation}
\begin{equation}\label{D'v-bound}
|D_{y'} v(y)| \le C \big(\varepsilon + (|y'|-r_0)_+^2 \big)^{-1/2}, \quad y \in C_{0,1},
\end{equation}
and
\begin{equation}\label{Dnv-bound}
|D_{n} v(y)| \le C\,\varepsilon^{-1}\big(\varepsilon+(|y'|-r_0)_+^2\big),  \quad y \in C_{0,1},  
\end{equation}
where $C>0$ depends only on $n$, $\gamma$, $\|h_1\|_{C^{2,\gamma}}$, $\|h_2\|_{C^{2,\gamma}}$, and $\|\varphi\|_{C^{1,1}}$.
\end{lemma}
\begin{proof}
First, we prove \eqref{Dv_uniform_bound}. For $y\in C_{0,r_0}$, by \eqref{fg_0}, \eqref{a_ij_formula}, and \eqref{a_ij_estimate}, we know that $(a^{ij}(y))$ is a diagonal matrix with $a^{nn}(y)=4\varepsilon$, $a^{ii}(y)=\varepsilon$ for $i=1,2,\ldots,n-1$, and $a^{ij}(y)=0$ for $i\neq j$. 
By linearity, $v$ satisfies \begin{equation*}
\left\{
\begin{aligned}
-D_i(c^{ij} D_j v(y)) &=0 \quad \mbox{in } C_{0,r_0},\\
c^{nj} D_j v(y) &= 0 \quad \mbox{on } \{(y',y_n)\in \mathbb{R}^n: ~|y'|<r_0,\; y_n=\pm \varepsilon\},
\end{aligned}
\right.
\end{equation*}
where $(c^{ij})$ is the constant diagonal matrix with $c^{nn}=4$, $c^{ii}=1$ for $i=1,2,\ldots,n-1$, and $c^{ij}=0$ for $i\neq j$. 
We now extend $v$ to the whole cylinder $Q_{r_0}:=\{y=(y', y_n)\in \mathbb{R}^n: \,|y'|<r_0\}$.
We take the even extension of $v$ with respect to ${y}_n = \varepsilon$ and then take the periodic extension in the $y_n$ axis, so that the period is equal to $4\varepsilon$. We still denote the function by $v$ after the extension. Then because of the conormal boundary conditions, $v$ satisfies
\begin{equation*}
-D_i(c^{ij} D_j v(y)) =0 \quad  \text{in} \; Q_{r_0}.
\end{equation*}
Then by classical gradient estimates for elliptic equations, we have
\begin{equation*}
    \|Dv\|_{L^\infty(Q_{3r_0/4})} \le \frac{C}{r_0}\,\|v\|_{L^\infty(Q_{r_0})},
\end{equation*}
and thus \eqref{Dv_uniform_bound} holds.

Next, by the change of variables in \eqref{x_to_y}, \eqref{D'v-bound} follows directly from \eqref{Du-bound} as
$$
\left| \frac{\partial x_i}{\partial y_j} \right| \le C \quad \mbox{for}~ i = 1,\ldots,n,~ j=1,\ldots, n-1.
$$
Since $\frac{\partial u}{\partial \nu} = 0$ on $\Gamma_+$ and $\Gamma_-$, we have, by
\eqref{fg_0}, \eqref{fg_1}, and \eqref{Du-bound} that
$$|D_n u(x)| \le  C\sum_{i = 1}^{n-1}(|x'|-r_0)_+ |D_i u|  \le C, \quad \forall x \in \Gamma_+ \cup \Gamma_-.$$
By the harmonicity of $D_n u$, the estimate \eqref{u_C1_outside}, and the maximum principle,
\begin{equation*}
|D_n u| \le C \quad \mbox{in}~~\Omega_{1},
\end{equation*}
and consequently, \eqref{Dnv-bound} holds.
The lemma is proved.
\end{proof}

We define
\begin{equation}
\label{v_bar_def}
\bar{v}(y') := \fint_{-\varepsilon}^\varepsilon v(y',y_n)\, dy_n.
\end{equation}
Then $\bar{v}$ satisfies
\begin{equation}\label{equation_v_bar}
\dv\big((\varepsilon+(|y'|-r_0)_+^2)D \bar v\big)= - \dv F  \quad\text{in}\,\,B_{1}\subset \bR^{n-1},
\end{equation}
with $\| \bar{v} \|_{L^\infty(B_{1})} \le C$, where
$$
F^i = \overline{a^{in}D_n v} + e^i D_i \bar v,
$$
$\overline{a^{in}\partial_n v}$ is the average of $a^{in}\partial_n v$ with respect to $y_n$ in $(-\varepsilon,\varepsilon)$. By \eqref{a_ij_estimate}, \eqref{D'v-bound}, and \eqref{Dnv-bound}, we know that
\begin{equation}\label{estimate_F}
\begin{aligned}
|F| &\le C\big((|y'|-r_0)_+(\varepsilon + (|y'|-r_0)_+^2) + (|y'|-r_0)_+^{2+\gamma}(\varepsilon+(|y'|-r_0)_+^2)^{-1/2}\big)\\
&\le C\varepsilon\,(|y'|-r_0)_+ +C\,(|y'|-r_0)_+^{1+\gamma},
\end{aligned}
\end{equation}
where $C>0$ depends only on $n$, $\gamma$, $\|h_1\|_{C^{2,\gamma}}$, $\|h_2\|_{C^{2,\gamma}}$, and $\|\varphi\|_{C^{1,1}}$.

Next, we write $x' \in \bR^{n-1}$ into the polar coordinate $(r,\xi) \in (0,\infty) \times \bS^{n-2}$, and take the spherical harmonic decomposition
\begin{equation}\label{sph-decom}
u(x',x_n) = \sum_{k=0}^\infty \sum_{i=1}^{N(k)} U_{k,i}(r,x_n)Y_{k,i}(\xi),
\end{equation}
where $\{Y_{k,i}\}_{i=1}^{N(k)}$ are the $k$-th degree normalized spherical harmonic functions on $\bS^{n-2}$ and 
\begin{equation}\label{def:Uki}
U_{k,i}(r,x_n)=\int_{\bS^{n-2}}u(r,\xi,x_n)Y_{k,i}(\xi)\,d\xi.
\end{equation}
Note that $\{Y_{k,i}\}$ is a basis for $L^2(\bS^{n-2})$.
Let $$\hat{\Omega}:=\{(r,x_n)\in \bR_+\times \bR:\,(r,0,\ldots,0,x_n)\in\widetilde\Omega\},$$  $$\hat\Upsilon_i:=\{(r,x_n)\in \bR_+\times \bR:\,(r,0,\ldots,0,x_n)\in\partial{\mathcal{D}_i}\}, \quad i=1,2,$$
and $\hat{\nu}$ denote the outward normal vector on $\hat\Omega$.
Recall that the $k$-th eigenvalue of $-\Delta_{\bS^{n-2}}$ is $k(k+n-3)$. Since $u$ is harmonic in $\widetilde{\Omega}$, by the separation of variables, we know that 
\begin{equation}\label{separation_pde}
    \Big(\frac{\partial^2}{\partial x_n^2}+\frac{\partial^2}{\partial r^2}+\frac{n-2}{r}\frac{\partial}{\partial r}-\frac{k(k+n-3)}{r^2}\Big)U_{k,i}=0 \quad \mbox{in } \hat{\Omega}
\end{equation}
holds for any $k\ge 0$ and $1\le i\le N(k)$.
Moreover, by \eqref{def:Uki}, $U_{k,i}$ also satisfies the Neumann boundary condition 
\begin{equation}\label{bcn:Uki}
    \partial_{\hat\nu} U_{k,i}=0 \quad \mbox{on }\hat\Upsilon_1\cup\hat\Upsilon_2.
\end{equation}
Now we prove that when $\varphi=f(r,x_n)Y_{k,i}(\xi)$, the solution $u$ to \eqref{equzero} has a single mode in the spherical harmonic decomposition \eqref{sph-decom}. More precisely, we have the following lemma.
\begin{lemma}\label{lem:single_decom}
    Let $u\in H^1(\widetilde\Omega)$ be the solution to \eqref{equzero}. If $\varphi=\hat{f}(r,x_n)\,Y_{k,i}(\xi)$, where $\hat{f}$ is a $C^{1,1}$ function, $k\ge 0$, and $1\le i\le N(k)$, then for any $x\in \widetilde\Omega$, we have
$$
u(x',x_n)=U_{k,i}(r,x_n)Y_{k,i}(\xi),
$$
where $Y_{k,i}$ and $U_{k,i}$ are defined as in \eqref{def:Uki}.
\end{lemma}
\begin{remark}
    Note that Lemma \ref{lem:single_decom} in particular implies that
\begin{equation*}
 u(x',x_n) =\left\{
 \begin{aligned}
     U_{1,1}(r,x_n)Y_{1,1}(\xi)&, &&\mbox{if } \varphi = x_1,\\
     U_{0,1}(r,x_n)&, && \mbox{if } \varphi = x_n.
 \end{aligned}
 \right.
\end{equation*}
\end{remark} 
\begin{proof}[Proof of Lemma \ref{lem:single_decom}]
    By \eqref{sph-decom}, it suffices to show that $U_{l,j}\equiv 0$ for any pair $\{l,j\}\neq \{k,i\}$. By \eqref{def:Uki} and the orthogonality of spherical harmonics, we know that
\begin{equation}\label{bcd:Uki}
    U_{l,j}=0 \quad \mbox{on } \partial\hat{\Omega}\setminus(\hat\Upsilon_1\cup \hat\Upsilon_2\cup\{r=0\}).
\end{equation}   
We claim that
\begin{equation}\label{Ulj-d}
    U_{l,j}=0 \quad \mbox{on } \partial \hat\Omega\cap \{r=0\}.
\end{equation}
We prove the claim by considering two different cases $k\ge 1$ and $k=0$. 

For the case when $k\ge 1$, we define 
    $$
\tilde{u}(x)=\left\{
\begin{aligned}
   &\fint_{\bS^{n-2}}u(|x'|,\xi,x_n)\,d\xi, && \mbox{if } x\in \widetilde\Omega, \;x'\neq 0,\\
   &u(0,x_n), && \mbox{if } x\in \widetilde\Omega, \;x'= 0.
\end{aligned}
\right.
    $$
Then $\tilde{u}\in C(\widetilde\Omega)$ satisfies the mean-value property in $\widetilde{\Omega}$. Therefore, $\tilde{u}$ is also harmonic in $\widetilde\Omega$. (See, e.g., \cite{GT}*{Theorem 2.7}.) By the second line of \eqref{equzero}, 
\begin{equation*}
   \partial_\nu \tilde{u}=0  \quad \mbox{on}~\partial\mathcal{D}_{1}\cup\partial\mathcal{D}_{2}.
\end{equation*}
Since $\varphi=\hat{f}(r,x_n)\,Y_{k,i}(\xi)$, by the symmetry of $Y_{k,i}$ ($k\ge 1$), we also have 
\begin{equation*}
    \tilde{u}(x)=\hat{f}(|x'|,x_n)\fint_{\bS^{n-2}}Y_{k,i}(\xi)\,d\xi=0  \quad \mbox{on}~\partial\Omega.
\end{equation*}
By the maximum principle, we can conclude that
$$\tilde{u}\equiv 0 \quad \mbox{in } \widetilde\Omega.$$
In particular, we have 
$$u(0,x_n)=0 \quad \mbox{for any } (0,x_n)\in \widetilde\Omega.$$
Therefore, by \eqref{def:Uki}, \eqref{Ulj-d} holds.

Next we treat the case when $k=0$. Since $u$ is smooth in $\widetilde\Omega$, for any $(0,x_n)\in \widetilde\Omega$, by \eqref{sph-decom} and the symmetry of $Y_{k,i}$ ($k\ge 1$), we have 
$$
\begin{aligned}
    u(0,x_n)&=\lim_{\delta\to 0}\fint_{\bS^{n-2}}u(\delta,\xi,x_n)\,d\xi\\
    &=\lim_{\delta\to 0}\fint_{\bS^{n-2}}U_{0,1}(\delta,x_n)Y_{0,1}\,d\xi=Y_{0,1}U_{0,1}(0,x_n).
\end{aligned}
$$
Here we used the fact that $N(0)=1$ and $Y_{0,1}$ is a constant.
Thus by \eqref{sph-decom} again, we have
$$
\sum_{l=1}^\infty \sum_{j=1}^{N(l)} U_{l,j}(0,x_n)Y_{l,j}(\xi)=0 \quad \mbox{for any } (0,x_n)\in \widetilde\Omega.
$$
Therefore, by the orthogonality of $Y_{l,j}$, for any $l\ge 1$, $$U_{l,j}(0,x_n)=0$$
and \eqref{Ulj-d} holds. The claim is proved.

Note that by \eqref{def:Uki}, $U_{l,j}\in C^2(\hat\Omega)\cap C(\overline{\hat\Omega})$.
By \eqref{separation_pde}, \eqref{bcn:Uki}, \eqref{bcd:Uki}, \eqref{Ulj-d}, we know that $U_{l,j}$ solves the following boundary value problem
$$
\left\{
\begin{aligned}
    \Big(\frac{\partial^2}{\partial x_n^2}+\frac{\partial^2}{\partial r^2}+\frac{n-2}{r}\frac{\partial}{\partial r}-\frac{l(l+n-3)}{r^2}\Big)U_{l,j}&=0 && \mbox{in }\hat{\Omega},\\
     \partial_{\hat\nu} U_{l,j}&=0 &&\mbox{on }\hat\Upsilon_1\cup \hat\Upsilon_2, \\
     U_{l,j}&=0 &&\mbox{on } \partial\hat{\Omega}\setminus(\hat\Upsilon_1\cup \hat\Upsilon_2).
\end{aligned}
\right.
$$
By the maximum principle, it follows that for any $\{l,j\}\neq \{k,i\}$
$$U_{l,j}\equiv 0 \quad \mbox{in } \hat\Omega.$$ 
The lemma is proved.
\end{proof}
Recall the discussion in Section \ref{sec1.1}, it suffices to prove Theorem \ref{main-thm} for the cases where $u$ contains a single nonzero mode and where $u$ contains a single zero mode, which will be treated in Sections \ref{sec3} and \ref{sec4}, respectively.

\section{The case when \texorpdfstring{$u$}{u} has a single nonzero mode}\label{sec3}
In this section, we assume that $\varphi=\hat{f}(r,x_n)Y_{k,i}(\xi)$, where $k\ge 1$ and $1\le i\le N(k)$.  Then, by Lemma \ref{lem:single_decom},
\begin{equation}\label{nonzero-mode}
u=U_{k,i}(r,x_n)Y_{k,i}(\xi).
\end{equation}
Then from \eqref{v_bar_def} and the fact that $h_1$, $h_2$ are radial, we know that
$$
\bar{v}(y') = V(r) Y_{k,i}(\xi).
$$
From \eqref{equation_v_bar}, we deduce that $V(r)$ satisfies the ODE
\begin{equation}\label{ODE1}
V''(r) + b(r)V'(r) - \frac{k(k+n-3)}{r^2}V(r) = H(r),
\end{equation}
with $V(0) = 0$, where
\begin{equation}\label{b_expression}
b(r) = \frac{n-2}{r} + \frac{2(r-r_0)_+}{\varepsilon + (r-r_0)_+^2}
\end{equation}
and
\begin{align*}
H(r) =& \int_{\bS^{n-2}} \frac{(\dv F) Y_{k,i}(\xi)}{\varepsilon+(r-r_0)_+^2} \, d\xi \\
=& \int_{\bS^{n-2}} \frac{D_r F_r + \frac{1}{r} D_\xi F_\xi}{\varepsilon + (r-r_0)_+^2} Y_{k,i}(\xi) \, d\xi\\
=& \partial_r \left(\int_{\bS^{n-2}} \frac{F_r}{\varepsilon + (r-r_0)_+^2} Y_{k,i}(\xi) \, d\xi \right)\\ 
&+ \int_{\bS^{n-2}} \frac{2(r-r_0)_+F_r Y_{k,i}}{(\varepsilon+(r-r_0)_+^2)^2} - \frac{F_\xi D_\xi Y_{k,i}}{r(\varepsilon + (r-r_0)_+^2)}\, d\xi\\
=&: A'(r) + B(r), \quad 0 < r < 1.
\end{align*}
Now we estimate $A(r)$ and $B(r)$.
By \eqref{estimate_F} and the fact that $\gamma\in(0,1)$, for $r\in(0,1)$,
we have
\begin{equation}\label{est:A1}
\begin{aligned}
    &|A(r)|\le \frac{C\,|F|}{\varepsilon + (r-r_0)_+^2}\le \frac{C\,\varepsilon\,(r-r_0)_+}{\varepsilon + (r-r_0)_+^2}+\frac{C\,(r-r_0)_+^{1+\gamma}}{\varepsilon + (r-r_0)_+^2
    }\\
    &\le \frac{C\,(r-r_0)_+}{(\varepsilon + (r-r_0)_+^2)^{1/2}}+\frac{C\,(r-r_0)_+^\gamma}{(\varepsilon + (r-r_0)_+^2)^{1/2}}\le
    \frac{C\,(r-r_0)_+^\gamma}{(\varepsilon + (r-r_0)_+^2)^{1/2}}.
\end{aligned}
\end{equation}
Here and throughout this section, we use $C$ to denote a constant depending only on $n$, $r_0$, $k$, $\gamma$, $\|h_1\|_{C^{2,\gamma}}$, $\|h_2\|_{C^{2,\gamma}}$, and $\|\hat{f}\|_{C^{1,1}}$, which may differ from line to line, unless otherwise specified.
Similarly, by the estimates in \eqref{est:A1}, for $r\in(0,1)$, it also holds that
\begin{equation}\label{est:B1}
\begin{aligned}
   &|B(r)|\le \frac{C\,|F|}{\varepsilon + (r-r_0)_+^2}\cdot \frac{(r-r_0)_+}{\varepsilon + (r-r_0)_+^2}+\frac{C\,|F|}{r(\varepsilon + (r-r_0)_+^2)}
   \\&\le \frac{C\,(r-r_0)_+^\gamma}{(\varepsilon + (r-r_0)_+^2)^{1/2}}\cdot \frac{1}{(\varepsilon + (r-r_0)_+^2)^{1/2}}+\frac{C\,(r-r_0)_+^\gamma}{r(\varepsilon + (r-r_0)_+^2)^{1/2}}.
\end{aligned}
\end{equation}
Note that when $r\le r_0$, the right-hand side of \eqref{est:B1} is exactly zero. Therefore, we can estimate
$$
\begin{aligned}
&|B(r)|\le \frac{C\,(r-r_0)_+^\gamma}{\varepsilon + (r-r_0)_+^2}+\frac{C\,(r-r_0)_+^\gamma}{r(\varepsilon + (r-r_0)_+^2)^{1/2}}
\\&\le \frac{C\,(r-r_0)_+^\gamma}{\varepsilon + (r-r_0)_+^2}+\frac{C\,(r-r_0)_+^\gamma}{r_0(\varepsilon + (r-r_0)_+^2)^{1/2}}\le \frac{C\,(r-r_0)_+^\gamma}{\varepsilon + (r-r_0)_+^2}.
\end{aligned}
$$
In summary, we have
\begin{equation}\label{estimate_AB}
|A(r)| \le \frac{C\,(r-r_0)_+^\gamma}{(\varepsilon + (r-r_0)_+^2)^{1/2}}, \quad |B(r)| \le \frac{C\,(r-r_0)_+^\gamma}{\varepsilon + (r-r_0)_+^2} , \quad 0 < r < 1.
\end{equation}

\begin{lemma}
\label{ODE_lemma}
For $\varepsilon > 0$, $n \ge 2$, $k\ge 1$, $r_0\in(0,1/2)$, there exists a unique solution $h \in C([0,1]) \cap C^{2,1}((0,1])$ of
\begin{equation*}
Lh :=h''(r) + b(r) h'(r) - \frac{k(k+n-3)}{r^2} h(r) = 0, \quad 0 < r <1
\end{equation*}
satisfying $h(0) = 0$ and $h(1) = 1$, where $b(r)$ is given as in \eqref{b_expression}.
Moreover, there exists a positive constant $C$ depending only on $n$, and $r_0$, such that
\begin{equation}
\label{h_bounds}
r^k  < h(r) < 1, \quad |h'(r)| \le C  \quad \mbox{for}~~0 < r < 1.
\end{equation}
\end{lemma}

\begin{proof}
Because of the singularity of $L$ at $r = 0$, we construct the solution $h$ through the approximating process as in the proofs of \cite{DLY}*{Lemma 3.1}. For $0 < a < r_0$, let $h_a \in C^2([a,1])$ be the solution of $Lh_a = 0$ in $(a, 1)$ satisfying $h_a(a) = a^k$ and $h_a(1) = 1$. By direct computations, we have
\begin{equation*}
Lr^k =  \frac{2k(r-r_0)_+}{\varepsilon + (r-r_0)_+^2} r^{k-1}\ge 0 \quad\quad \mbox{for}~r\in(0,1),
\end{equation*}
Note that $a^k=h_a(a)< 1$ since $k\ge 1$.
By the strong maximum principle,
$$
r^k< h_a(r) < 1, \quad a < r < 1.
$$
On the other hand, since $Lr^k=0$ for $r\in(a,r_0)$, $h_a(a)=a^k$, $h_a(r_0)<1=r_0^k/r_0^k$, by the comparison principle, we have 
$$h_a(r)\le r^k/r_0^k, \quad a<r<r_0.$$
Sending $a \to 0$, along a subsequence, $h_a \to h$ in $C^2_{\text{loc}}((0,1])$ for some $h \in C([0,1]) \cap C^\infty((0,1])$ satisfying $r^k \le h(r) \le 1$, $Lh = 0$ in $(0,1)$, and $h(0) = 0$. By the strong maximum principle applied to $h(r) - r^k$ and $1 - h(r)$,
$$
r ^k< h(r) < 1, \quad 0 < r < 1.
$$
We then prove that $|h'(r)| \le C$. Since $h$ satisfies
$$
h''(r) + \frac{n-2}{r} h'(r) -\frac{k(k+n-3)}{r^2} h(r) = 0, \quad 0 < r < r_0,
$$
and $h(0) = 0$, we know that
$$
h(r) = C_1r^k, \quad 0 < r < r_0
$$
for some positive constant $C_1$. The upper bound $h(r_0) \le 1$ implies that $C_1 \le 1/r_0^k$, therefore,
\begin{equation}\label{h'1}
h'(r) = C_1k\,r^{k-1} \le \frac{k}{r_0}, \quad 0 < r < r_0.
\end{equation}
For $r \ge r_0/2$, we can use integrating factor to express $h'(r)$ in terms of $h(r)$, that is
\begin{equation}\label{h'2}
\begin{aligned}
h'(r) = e^{-\int_{r_0/2}^r b(s) \, ds} h'\Big(\frac{r_0}{2}\Big) + k(k+n-3) e^{-\int_{r_0/2}^r b(s) \, ds} \int_{r_0/2}^r e^{\int_{r_0/2}^t b(s) \, ds} \frac{h(t)}{t^2} \, dt.
\end{aligned}
\end{equation}
Since $b \in  C^{0,1}((0,1])$, we can see that $h \in  C^{2,1}((0,1])$. Next we compute for $t \ge r_0/2$,
\begin{equation}\label{integral_b}
\begin{aligned}
\int_{r_0/2}^t b(s) \, ds =& \int_{r_0/2}^t \frac{n-2}{s} + \frac{2(s-r_0)_+}{\varepsilon + (s - r_0)_+^2} \, ds\\
=& (n-2) \Big( \log t - \log (r_0/2) \Big) + \log \left( \frac{\varepsilon + (t - r_0)_+^2}{\varepsilon} \right).
\end{aligned}
\end{equation}
Therefore, for $t \in [r_0/2,1]$ and $r \in [r_0/2,1]$,
$$
e^{\int_{r_0/2}^t b(s) \, ds} \le C \frac{\varepsilon + (t - r_0)_+^2}{\varepsilon}, \quad e^{-\int_{r_0/2}^r b(s) \, ds} \le  \frac{C\varepsilon}{\varepsilon + (r - r_0)_+^2}.
$$
Plugging into \eqref{h'2}, and using \eqref{h'1} and $h \le 1$, we can conclude $|h'| \le C$.
\end{proof}

Let $h(r)$ be the unique solution as in Lemma \ref{ODE_lemma}. 
We will solve \eqref{ODE1} by the method of reduction of order to obtain the desired estimates. Let $w(r) = V(r)/h(r)$. Then $w(r)$ satisfies
$$
w''(r) + \left( \frac{2h'(r)}{h(r)} + b(r) \right) w'(r) = \frac{H(r)}{h(r)}, \quad 0 < r < 1.
$$
Using the integrating factor again, we can write
\begin{equation}\label{w'}
\begin{aligned}
w'(r) =& e^{-\int_{r_0/2}^r \left(\frac{2h'(s)}{h(s)} + b(s)\right) \, ds}w'\Big(\frac{r_0}{2}\Big)\\
&+  e^{-\int_{r_0/2}^r \left(\frac{2h'(s)}{h(s)} + b(s)\right) \, ds} \int_{r_0/2}^r  e^{\int_{r_0/2}^t \left(\frac{2h'(s)}{h(s)} + b(s)\right) \, ds} \frac{A'(t) + B(t)}{h(t)} \, dt.
\end{aligned}
\end{equation}
Similar as in \eqref{integral_b}, we can use \eqref{h_bounds} to estimate
\begin{equation}\label{integrating_factor}
\begin{aligned}
e^{\int_{r_0/2}^t \left(\frac{2h'(s)}{h(s)} + b(s)\right) \, ds} &\le C \frac{\varepsilon + (t - r_0)_+^2}{\varepsilon}, \\ 
e^{-\int_{r_0/2}^r \left(\frac{2h'(s)}{h(s)} + b(s)\right) \, ds} &\le  \frac{C\varepsilon}{\varepsilon + (r - r_0)_+^2}.    
\end{aligned}
\end{equation}
Now we estimate the right-hand side of \eqref{w'}. For the term containing $A'$, we use integration by parts to obtain
\begin{align*}
&\int_{r_0/2}^r  e^{\int_{r_0/2}^t \left(\frac{2h'(s)}{h(s)} + b(s)\right) \, ds} \frac{A'(t)}{h(t)} \, dt\\
&= \frac{A(r)}{h(r)}e^{\int_{r_0/2}^r \left(\frac{2h'(s)}{h(s)} + b(s)\right) \, ds}\\
&\quad - \int_{r_0/2}^r  e^{\int_{r_0/2}^t \left(\frac{2h'(s)}{h(s)} + b(s)\right) \, ds} \frac{A(t)}{h(t)} \left( \frac{2h'(t)}{h(t)} + b(t) \right) \, dt\\
&\quad + \int_{r_0/2}^r  e^{\int_{r_0/2}^t \left(\frac{2h'(s)}{h(s)} + b(s)\right) \, ds} \frac{A(t)h'(t)}{h(t)^2}\, dt\\
&=: \rm{I}(r) + \rm{II}(r) + \rm{\rm{III}}(r).
\end{align*}
By \eqref{estimate_AB} and \eqref{integrating_factor}, we can estimate
\begin{equation}\label{est:A-1}
\begin{aligned}
    &\left| e^{-\int_{r_0/2}^r \left(\frac{2h'(s)}{h(s)} + b(s)\right) \, ds} \rm{I}(r) \right|\\
    &\le \frac{C\varepsilon}{\varepsilon + (r - r_0)_+^2} \cdot \frac{(r-r_0)_+^\gamma}{(\varepsilon + (r-r_0)_+^2)^{1/2}} \cdot \frac{\varepsilon + (r - r_0)_+^2}{\varepsilon}\\
    &\le  C\,\big(\varepsilon + (r-r_0)_+^2\big)^{\frac{\gamma-1}{2}}.
\end{aligned}
\end{equation}
Recall the definition of $b$ in \eqref{b_expression}, by \eqref{estimate_AB} and \eqref{integrating_factor} again, we can estimate
\begin{align*}
    |\rm{II}(r)| \le & C \int_{r_0/2}^r \frac{\varepsilon + (t - r_0)_+^2}{\varepsilon}\cdot \frac{(t-r_0)_+^\gamma}{(\varepsilon + (t-r_0)_+^2)^{1/2}} \cdot \Big( 1 +  \frac{(t-r_0)_+}{\varepsilon + (t-r_0)_+^2} \Big) \, dt\\
    \le & \frac{C}{\varepsilon} \int_{r_0/2}^r (t-r_0)_+^\gamma \, dt\\
    \le & \frac{C}{\varepsilon} (r-r_0)_+^{\gamma+1}.
\end{align*}
Therefore,
\begin{equation}\label{est:A-2}
\begin{aligned}
    \left| e^{-\int_{r_0/2}^r \left(\frac{2h'(s)}{h(s)} + b(s)\right) \, ds} \rm{II}(r) \right| \le & \frac{C\varepsilon}{\varepsilon + (r - r_0)_+^2} \cdot \frac{1}{\varepsilon} (r-r_0)_+^{\gamma+1}\\
    \le & C\,\big(\varepsilon + (r-r_0)_+^2\big)^{\frac{\gamma-1}{2}}.
\end{aligned}
\end{equation}
Similarly, we have
\begin{equation}\label{est:A-3}
\begin{aligned}
     & \left| e^{-\int_{r_0/2}^r \left(\frac{2h'(s)}{h(s)} + b(s)\right) \, ds} \rm{III}(r) \right|\\
     &\le  \frac{C\varepsilon}{\varepsilon + (r - r_0)_+^2} \int_{r_0/2}^r \frac{\varepsilon + (t - r_0)_+^2}{\varepsilon}\cdot \frac{(t-r_0)_+^\gamma}{(\varepsilon + (t-r_0)_+^2)^{1/2}} \, dt\\
     &\le  \frac{C\varepsilon}{\varepsilon + (r - r_0)_+^2} \cdot \frac{1}{\varepsilon} (r-r_0)_+^{\gamma+1}\\
     &\le  C\,\big(\varepsilon + (r-r_0)_+^2\big)^{\frac{\gamma-1}{2}}.
\end{aligned}
\end{equation}
Combining \eqref{est:A-1}-\eqref{est:A-3}, we have
\begin{equation}\label{est:A}
\left| e^{-\int_{r_0/2}^r \left(\frac{2h'(s)}{h(s)} + b(s)\right) \, ds} \int_{r_0/2}^r  e^{\int_{r_0/2}^t \left(\frac{2h'(s)}{h(s)} + b(s)\right) \, ds} \frac{A'(t)}{h(t)} \, dt \right| \le C\,\big(\varepsilon + (r-r_0)_+^2\big)^{\frac{\gamma-1}{2}}.
\end{equation}
Similarly, we can also estimate
\begin{equation}\label{est:B}
\begin{aligned}
    & \left| e^{-\int_{r_0/2}^r \left(\frac{2h'(s)}{h(s)} + b(s)\right) \, ds} \int_{r_0/2}^r  e^{\int_{r_0/2}^t \left(\frac{2h'(s)}{h(s)} + b(s)\right) \, ds} \frac{B(t)}{h(t)} \, dt \right|\\
     &\le  \frac{C\varepsilon}{\varepsilon + (r - r_0)_+^2} \int_{r_0/2}^r \frac{\varepsilon + (t - r_0)_+^2}{\varepsilon}\cdot \frac{(t-r_0)_+^\gamma}{\varepsilon + (t-r_0)_+^2} \, dt\\
     &\le  \frac{C}{\varepsilon + (r - r_0)_+^2} \cdot (r-r_0)_+^{\gamma+1}\\
    &\le  C\,\big(\varepsilon + (r-r_0)_+^2\big)^{\frac{\gamma-1}{2}}.
\end{aligned}
\end{equation}
Finally, by the definition of $w$, we have
$$
w'(r) = \frac{V'(r)}{h(r)} - \frac{h'(r)V(r)}{h^2(r)}.
$$
Therefore, $|w'(r_0/2)| \le C$ due to \eqref{Dv_uniform_bound} and \eqref{h_bounds}. Then from \eqref{w'}, \eqref{est:A}, and \eqref{est:B}, we conclude that
$$|w'(r)| \le C+C\,\big(\varepsilon + (r-r_0)_+^2\big)^{\frac{\gamma-1}2}, \quad  r\in [r_0/2, 1).$$  
This implies
\begin{equation}\label{V'-bound}
\begin{aligned}
|V'(r)| &=  |h'(r) w(r) + h(r) w'(r)| \le C+ C\,\big(\varepsilon + (r-r_0)_+^2\big)^{\frac{\gamma-1}2}\\ &\le C\,\big(\varepsilon + (r-r_0)_+^2\big)^{\frac{\gamma-1}2}, \quad r \in [r_0/2, 1).
\end{aligned}
\end{equation}
Now we are ready to prove Theorem \ref{main-thm} under the assumption \eqref{nonzero-mode}.
\subsection{Proof of Theorem \ref{main-thm} when  \texorpdfstring{$u$}{u} has a single nonzero mode}\label{sec3.1}
\begin{proof}
Without loss of generality, we may assume $\varepsilon<1/4$.
We prove \eqref{main-eq} using a bootstrap argument. Let $v$, $\bar{v}$, $V$ be defined as above. First, by the argument as above, from \eqref{V'-bound} we obtain that 
\begin{equation}\label{bar_v_bound}
    |D_{y'}\bar{v}(y')|\le C \,\big(\varepsilon + (|y'|-r_0)_+^2\big)^{-s_0+\frac{\gamma}{2}}, \quad r_0/2\le |y'|\le 1,
\end{equation}
where $s_0=1/2$.
Recall the definition of $C_{x,s}$ in \eqref{C_x_s}. By \eqref{Dnv-bound},
\begin{equation}
\label{v-v_bar}
|v(y', y_n) - \bar{v}(y')| \le 2\varepsilon \max_{y_n\in (- \varepsilon , \varepsilon)} |\partial_n v(y',y_n)| \le C (\varepsilon + (|y'|-r_0)_+^2) \quad \mbox{in}~~C_{0,1}.
\end{equation}
Let $x\in \Omega_1$. We denote $\eta=\frac{1}{4}(\varepsilon + (|x'|-r_0)_+^2)^{1/2}$.
By \eqref{Dv_uniform_bound} and the change of variables in \eqref{x_to_y}, we have 
$$
\| Du\|_{L^\infty (\Omega_{3r_0/4})} \le C.
$$
When $3r_0/4\le|x'|\le r_0+\sqrt\varepsilon$, 
by \eqref{bar_v_bound} and \eqref{v-v_bar}, we have
\begin{equation}
\begin{aligned}\label{v_L2_control}
\underset{\Omega_{x,\eta}}{\osc}~u =&  \underset{C_{x,\eta}}{\osc}~v \\
 \le& C\big( \|v-\bar{v}\|_{L^\infty(C_{x,\eta})} + \underset{C_{x,\eta}}{\osc}~\bar{v}\big)\\
 \le & C\varepsilon+C\sqrt{\varepsilon} \,\|D_{y'}\bar{v}\|_{L^\infty(C_{x,\eta})}\\
 \le& C \varepsilon^{-s_0+\frac{1+\gamma}2}.
\end{aligned}
\end{equation}
By \eqref{v_L2_control} and \eqref{Du-bound}, we have
$$
\| Du\|_{L^\infty (\Omega_{r_0+\sqrt\varepsilon}\setminus \Omega_{3r_0/4})} \le C\varepsilon^{-s_0+\frac{\gamma}2}.
$$
When $r_0+\sqrt\varepsilon \le|x'|\le 3/4$, by the triangle inequality, for $y\in C_{x,\eta}$ we have 
$$ 0<\frac{1}{2}(|x'|-r_0)<|x'|-r_0-\eta\le|y'|-r_0\le |x'|-r_0+\eta\le 5\eta,$$
and thus
$$\frac{1}{4}(\varepsilon + (|x'|-r_0)_+^2)\le \varepsilon + (|y'|-r_0)_+^2\le 4(\varepsilon + (|x'|-r_0)_+^2).$$
Then by \eqref{bar_v_bound} and \eqref{v-v_bar} again, we have
$$
\underset{\Omega_{x,\eta}}{\osc}~u\le C (\varepsilon + (|x'|-r_0)_+^2)^{-s_0+\frac{1+\gamma}2}, 
$$
which implies, by \eqref{Du-bound}, that
$$
| Du(x)| \le C (\varepsilon + (|x'|-r_0)_+^2)^{-s_0+\frac{\gamma}2}, \quad x\in \Omega_{3/4}\setminus \Omega_{r_0+2\sqrt{\varepsilon}}.
$$
Therefore, we have improved the upper bound $$
|D u(x)| \le C\,(\varepsilon + (|x'|-r_0)_+^2)^{-s_0}$$ 
to $$
|Du(x)| \le C\,(\varepsilon + (|x'|-r_0)_+^2)^{-s_0+\gamma/2}$$ 
for $x\in \Omega_{3/4}$ and $s_0=1/2$. Then we take $s_1 = s_0 - \frac{\gamma}{2}$ and repeat the argument above. In the general step, if we have already obtained an upper bound
\begin{equation}
    \label{eq12.40}
|D u(x)| \le C\,(\varepsilon + (|x'|-r_0)_+^2)^{-s}\,\, \text{with } s\in(0,1/2), \,\, \text{for } x\in \Omega_{3/4},
\end{equation}
then instead of \eqref{estimate_F}, we have
$$
|F| \le C\varepsilon\,(|y'|-r_0)_+ +C\,(|y'|-r_0)_+^{2+\gamma-2s}.
$$
Consequently, the estimates in \eqref{estimate_AB} can be replaced with
$$
|A(r)| \le \frac{C\,(r-r_0)_+^{1-(2s-\gamma)_+}}{(\varepsilon + (r-r_0)_+^2)^{1/2}}, \quad |B(r)| \le \frac{C\,(r-r_0)_+^{1-(2s-\gamma)_+}}{\varepsilon + (r-r_0)_+^2} , \quad 0 < r < 1.
$$
Therefore, by following the argument above, one can obtain 
$$
|V'(r)| \le C+ C\,\big(\varepsilon + (r-r_0)_+^2\big)^{\frac{-(2s-\gamma)_+}2} \le C\,\big(\varepsilon + (r-r_0)_+^2\big)^{-\hat{s}}, \quad r \in [r_0/2, 1),
$$
where $\hat{s}=(s-\gamma/2)_+$.
Then we can improve the upper bound 
\eqref{eq12.40} to 
$$
|Du(x)| \le C\,(\varepsilon + (|x'|-r_0)_+^2)^{-\hat{s}} \quad \text{for } x\in \Omega_{3/4},
$$ 
as above.
After repeating the argument finite times and using the classical estimates \eqref{u_C1_outside}, we obtain the uniform bound \eqref{main-eq}. The proof is completed.
\end{proof}

\section{The case when \texorpdfstring{$u$}{u} has a single zero mode}\label{sec4}
In this section, we assume that $\varphi=\hat{f}(r,x_n)Y_{0,1}(\xi)$ and thus by Lemma \ref{lem:single_decom},
\begin{equation*}
    u=U_{0,1}(r,x_n).
\end{equation*}
Then from \eqref{v_bar_def}, we know that
$$
\bar{v}(y') = V(r)
$$
is a radial function. Similarly as in Section \ref{sec3}, we can deduce from \eqref{equation_v_bar} that $V(r)$ satisfies the ODE
\begin{equation*}
V''(r) + b(r)V'(r) = H(r),
\end{equation*}
 where $b(r)$ is given as in \eqref{b_expression}, and
\begin{align*}
H(r) =& \int_{\bS^{n-2}} \frac{\dv F}{\varepsilon+(r-r_0)_+^2} \, d\xi \\
=& \int_{\bS^{n-2}} \frac{D_r F_r + \frac{1}{r} D_\xi F_\xi}{\varepsilon + (r-r_0)_+^2} \, d\xi\\
=& \partial_r \left(\int_{\bS^{n-2}} \frac{F_r}{\varepsilon + (r-r_0)_+^2}  \, d\xi \right) + \int_{\bS^{n-2}} \frac{2(r-r_0)_+F_r}{(\varepsilon+(r-r_0)_+^2)^2}\\
=&: A'(r) + B(r), \quad 0 < r < 1,
\end{align*}
where $A$ and $B$ satisfy the same estimates \eqref{estimate_AB} and $F_\xi=0$. We can use integrating factor again to obtain
\begin{equation}\label{if-2}
V'(r) = e^{-\int_{r_0/2}^r b(s) \, ds} V'\Big(\frac{r_0}{2}\Big) + e^{-\int_{r_0/2}^r b(s) \, ds} \int_{r_0/2}^r e^{\int_{r_0/2}^t b(s) \, ds} (A'(t) + B(t)) \, dt.
\end{equation}
The right-hand side of \eqref{if-2} can be estimated in the same way as \eqref{w'} by setting $h \equiv 1$. Therefore, from \eqref{estimate_AB} we obtain
\begin{equation}\label{V'-bound-2}
    | V'(r)|  \le C+C\,\big(\varepsilon + (r-r_0)_+^2\big)^{\frac{\gamma-1}2} \le C\,\big(\varepsilon + (r-r_0)_+^2\big)^{\frac{\gamma-1}2}, \quad  r\in [r_0/2, 1),
\end{equation}
where $C>0$ depends only on $n$, $r_0$, $\gamma$, $\|h_1\|_{C^{2,\gamma}}$, $\|h_2\|_{C^{2,\gamma}}$, and $\|\hat{f}\|_{C^{1,1}}$.
\subsection{Proof of Theorem \ref{main-thm} when \texorpdfstring{$u$}{u} has a single zero mode}
With the estimate \eqref{V'-bound-2} in place of \eqref{V'-bound}, one can follow the same bootstrap argument as in Section \ref{sec3.1} to prove the uniform bound \eqref{main-eq}. Therefore, we omit the details.

\bibliographystyle{amsplain}

\begin{bibdiv}
\begin{biblist}

\bib{ACKLY}{article}{
      author={Ammari, H.},
      author={Ciraolo, G.},
      author={Kang, H.},
      author={Lee, H.},
      author={Yun, K.},
       title={Spectral analysis of the {N}eumann-{P}oincar\'{e} operator and
  characterization of the stress concentration in anti-plane elasticity},
        date={2013},
        ISSN={0003-9527},
     journal={Arch. Ration. Mech. Anal.},
      volume={208},
      number={1},
       pages={275\ndash 304},
  url={https://doi-org.proxy.libraries.rutgers.edu/10.1007/s00205-012-0590-8},
      review={\MR{3021549}},
}

\bib{AKLLL}{article}{
      author={Ammari, H.},
      author={Kang, H.},
      author={Lee, H.},
      author={Lee, J.},
      author={Lim, M.},
       title={Optimal estimates for the electric field in two dimensions},
        date={2007},
        ISSN={0021-7824},
     journal={J. Math. Pures Appl. (9)},
      volume={88},
      number={4},
       pages={307\ndash 324},
  url={https://doi-org.proxy.libraries.rutgers.edu/10.1016/j.matpur.2007.07.005},
      review={\MR{2384571}},
}

\bib{AKL}{article}{
      author={Ammari, H.},
      author={Kang, H.},
      author={Lim, M.},
       title={Gradient estimates for solutions to the conductivity problem},
        date={2005},
        ISSN={0025-5831},
     journal={Math. Ann.},
      volume={332},
      number={2},
       pages={277\ndash 286},
  url={https://doi-org.proxy.libraries.rutgers.edu/10.1007/s00208-004-0626-y},
      review={\MR{2178063}},
}

\bib{BASL}{article}{
      author={Babu\v{s}ka, I.},
      author={Andersson, B.},
      author={Smith, P.J.},
      author={Levin, K.},
       title={Damage analysis of fiber composites. {I}. {S}tatistical analysis
  on fiber scale},
        date={1999},
        ISSN={0045-7825},
     journal={Comput. Methods Appl. Mech. Engrg.},
      volume={172},
      number={1-4},
       pages={27\ndash 77},
  url={https://doi-org.proxy.libraries.rutgers.edu/10.1016/S0045-7825(98)00225-4},
      review={\MR{1685902}},
}

\bib{BLY1}{article}{
      author={Bao, E.},
      author={Li, Y.Y.},
      author={Yin, B.},
       title={Gradient estimates for the perfect conductivity problem},
        date={2009},
        ISSN={0003-9527},
     journal={Arch. Ration. Mech. Anal.},
      volume={193},
      number={1},
       pages={195\ndash 226},
  url={https://doi-org.proxy.libraries.rutgers.edu/10.1007/s00205-008-0159-8},
      review={\MR{2506075}},
}

\bib{BLY2}{article}{
      author={Bao, E.},
      author={Li, Y.Y.},
      author={Yin, B.},
       title={Gradient estimates for the perfect and insulated conductivity
  problems with multiple inclusions},
        date={2010},
        ISSN={0360-5302},
     journal={Comm. Partial Differential Equations},
      volume={35},
      number={11},
       pages={1982\ndash 2006},
  url={https://doi-org.proxy.libraries.rutgers.edu/10.1080/03605300903564000},
      review={\MR{2754076}},
}

\bib{BV}{article}{
      author={Bonnetier, E.},
      author={Vogelius, M.},
       title={An elliptic regularity result for a composite medium with
  ``touching'' fibers of circular cross-section},
        date={2000},
        ISSN={0036-1410},
     journal={SIAM J. Math. Anal.},
      volume={31},
      number={3},
       pages={651\ndash 677},
  url={https://doi-org.proxy.libraries.rutgers.edu/10.1137/S0036141098333980},
      review={\MR{1745481}},
}

\bib{CLX}{article}{
      author={Chen, Y.},
      author={Li, H.},
      author={Xu, L.},
       title={Optimal gradient estimates for the perfect conductivity problem
  with {$C^{1, \alpha}$} inclusions},
        date={2021},
        ISSN={0294-1449,1873-1430},
     journal={Ann. Inst. H. Poincar\'e{} C Anal. Non Lin\'eaire},
      volume={38},
      number={4},
       pages={953\ndash 979},
         url={https://doi.org/10.1016/j.anihpc.2020.09.009},
      review={\MR{4266231}},
}

\bib{DLY2}{article}{
      author={Dong, H.},
      author={Li, Y.Y.},
      author={Yang, Z.},
       title={Gradient estimates for the insulated conductivity problem: {T}he
  non-umbilical case},
        date={2024},
        ISSN={0021-7824,1776-3371},
     journal={J. Math. Pures Appl. (9)},
      volume={189},
       pages={103587},
         url={https://doi.org/10.1016/j.matpur.2024.06.002},
      review={\MR{4779390}},
}

\bib{DLY}{article}{
      author={Dong, H.},
      author={Li, Y.Y.},
      author={Yang, Z.},
       title={Optimal gradient estimates of solutions to the insulated
  conductivity problem in dimension greater than two},
        date={2025},
        ISSN={1435-9855,1435-9863},
     journal={J. Eur. Math. Soc. (JEMS)},
      volume={27},
      number={8},
       pages={3275\ndash 3296},
         url={https://doi.org/10.4171/jems/1432},
      review={\MR{4911712}},
}

\bib{DYZ24}{article}{
      author={Dong, H.},
      author={Yang, Z.},
      author={Zhu, H.},
       title={Gradient estimates for the conductivity problem with imperfect
  bonding interfaces},
        date={2024},
         url={https://arxiv.org/abs/2409.05652},
        note={arXiv:2409.05652},
}


\bib{fukushima2024finiteness}{article}{
      author={Fukushima, S.},
      author={Ji, Y.-G.},
      author={Kang, H.},
      author={Li, X.},
       title={Finiteness of the stress in presence of closely located
  inclusions with imperfect bonding},
        date={2024},
     journal={Math. Ann.},
        note={DOI 10.1007/s00208-024-02968-9},
}

\bib{GT}{book}{
      author={Gilbarg, D.},
      author={Trudinger, N.},
       title={Elliptic partial differential equations of second order},
      series={Classics in Mathematics},
   publisher={Springer-Verlag, Berlin},
        date={2001},
        ISBN={3-540-41160-7},
      review={\MR{1814364}},
}

\bib{Kang}{incollection}{
      author={Kang, H.},
       title={Quantitative analysis of field concentration in presence of
  closely located inclusions of high contrast},
        date={[2023] \copyright2023},
   booktitle={I{CM}---{I}nternational {C}ongress of {M}athematicians. {V}ol. 7.
  {S}ections 15--20},
   publisher={EMS Press, Berlin},
       pages={5680\ndash 5699},
      review={\MR{4680458}},
}

\bib{li2024optimalgradientestimatesinsulated}{article}{
      author={Li, H.},
      author={Zhao, Y.},
       title={Optimal gradient estimates for the insulated conductivity problem
  with general convex inclusions case},
        date={2024},
         url={https://arxiv.org/abs/2404.17201},
        note={arXiv:2404.17201},
}

\bib{LN}{article}{
      author={Li, Y.Y.},
      author={Nirenberg, L.},
       title={Estimates for elliptic systems from composite material},
        date={2003},
        ISSN={0010-3640},
     journal={Comm. Pure Appl. Math.},
      volume={56},
      number={7},
       pages={892\ndash 925},
         url={https://doi-org.proxy.libraries.rutgers.edu/10.1002/cpa.10079},
      review={\MR{1990481}},
}

\bib{LV}{article}{
      author={Li, Y.Y.},
      author={Vogelius, M.},
       title={Gradient estimates for solutions to divergence form elliptic
  equations with discontinuous coefficients},
        date={2000},
        ISSN={0003-9527},
     journal={Arch. Ration. Mech. Anal.},
      volume={153},
      number={2},
       pages={91\ndash 151},
  url={https://doi-org.proxy.libraries.rutgers.edu/10.1007/s002050000082},
      review={\MR{1770682}},
}

\bib{LY2}{article}{
      author={Li, Y.Y.},
      author={Yang, Z.},
       title={Gradient estimates of solutions to the insulated conductivity
  problem in dimension greater than two},
        date={2023},
        ISSN={0025-5831},
     journal={Math. Ann.},
      volume={385},
      number={3-4},
       pages={1775\ndash 1796},
         url={https://doi.org/10.1007/s00208-022-02368-x},
      review={\MR{4566706}},
}

\bib{We}{article}{
      author={Weinkove, B.},
       title={The insulated conductivity problem, effective gradient estimates
  and the maximum principle},
        date={2023},
        ISSN={0025-5831},
     journal={Math. Ann.},
      volume={385},
      number={1-2},
       pages={1\ndash 16},
         url={https://doi.org/10.1007/s00208-021-02314-3},
      review={\MR{4542709}},
}

\bib{Y1}{article}{
      author={Yun, K.},
       title={Estimates for electric fields blown up between closely adjacent
  conductors with arbitrary shape},
        date={2007},
        ISSN={0036-1399},
     journal={SIAM J. Appl. Math.},
      volume={67},
      number={3},
       pages={714\ndash 730},
         url={https://doi-org.proxy.libraries.rutgers.edu/10.1137/060648817},
      review={\MR{2300307}},
}

\bib{Y2}{article}{
      author={Yun, K.},
       title={Optimal bound on high stresses occurring between stiff fibers
  with arbitrary shaped cross-sections},
        date={2009},
        ISSN={0022-247X},
     journal={J. Math. Anal. Appl.},
      volume={350},
      number={1},
       pages={306\ndash 312},
  url={https://doi-org.proxy.libraries.rutgers.edu/10.1016/j.jmaa.2008.09.057},
      review={\MR{2476915}},
}

\end{biblist}
\end{bibdiv}

\end{document}